\documentclass{amsart}

\usepackage{amsaddr}
\usepackage{amssymb}
\usepackage{enumerate}
\usepackage{hyperref}

\newtheorem{formula}{}[section]
\newtheorem{corollary}[formula]{Corollary}
\newtheorem{lemma}[formula]{Lemma}
\newtheorem{theorem}[formula]{Theorem}
\newtheorem{proposition}[formula]{Proposition}

\DeclareMathOperator{\aut}{Aut}
\DeclareMathOperator{\alt}{Alt}
\DeclareMathOperator{\sym}{Sym}
\DeclareMathOperator{\GaL}{{\rm \Gamma}L}
\DeclareMathOperator{\GL}{GL}
\DeclareMathOperator{\out}{Out}
\DeclareMathOperator{\PGaL}{P\Gamma L}
\DeclareMathOperator{\soc}{Soc}
\DeclareMathOperator{\F}{\mathbb{F}}
\DeclareMathOperator{\X}{\mathfrak{X}}
\DeclareMathOperator{\C}{\mathcal{C}}

\DeclareMathOperator{\krel}{\mathit{k}-rel}

\begin{document}

\title[Closures of permutation groups with restricted composition factors]{Closures of permutation groups with restricted nonabelian composition factors}
\author{Ilia Ponomarenko}
\address{Steklov Institute of Mathematics at St. Petersburg, Russia}
\email{inp@pdmi.ras.ru}

\author{Saveliy V. Skresanov}
\address{Novosibirsk State University, Novosibirsk 630090, Russia\\
Sobolev Institute of Mathematics, Novosibirsk 630090, Russia}
\email{skresan@math.nsc.ru}

\author{Andrey V. Vasil'ev}
\address{Novosibirsk State University, Novosibirsk 630090, Russia\\
Sobolev Institute of Mathematics, Novosibirsk 630090, Russia}
\email{vasand@math.nsc.ru}

\thanks{}
\date{}

\begin{abstract}
	Given a permutation group \( G \) on a finite set \( \Omega \), let \( G^{(k)} \) denote the \( k \)-closure of \( G \),
	that is, the largest permutation group on \( \Omega \) having the same orbits in the induced action on \( \Omega^k \) as \( G \).
	Recall that a group is \( \alt(d) \)-free if it does not contain a section isomorphic to the alternating group of degree~\( d \).
	Motivated by some problems in computational group theory, we prove that the \( k \)-closure of an \( \alt(d) \)-free group is again
	\( \alt(d) \)-free for \( k \geq 4 \) and \( d \geq 25 \). 
\end{abstract}

\maketitle

\section{Introduction}

Throughout all groups and sets on which they act are assumed to be finite.

One of the motivations for the present paper is a computational problem in which one needs to find efficiently the automorphism group
$\aut(\mathfrak S)$ of a given set $\mathfrak S$ of relations (generally speaking of different arities). In the case when
all the relations of $\mathfrak S$ are binary, this problem is equivalent to the famous Graph Isomorphism Problem and can be solved by the
Babai algorithm~\cite{B2015} in quasipolynomial time in the size of $\mathfrak S$.  It is currently unknown whether the  automorphism group of a
graph  can be found in polynomial time. It is surprising, however, that apparently the problem does not become easier if a sufficiently large
subgroup of  $\aut(\mathfrak S)$ is known in advance. In the present paper, we take a step towards the latter problem by studying
the structure of the group $\aut(\mathfrak S)$ in the case when all the relations of~$\mathfrak S$  are of arity at least~$4$. To be more
precise, we recall briefly the method of invariant relations that was introduced and developed by Wielandt in the late 1960s
\cite{Wielandt1969}, and used in the above-mentioned Babai's paper in the framework of multidimensional coherent configurations.

The method of invariant relations was  considered by Wielandt  as one of the main tools for studying  actions of a group on a set
\cite{Wielandt1969a}. This method is based on the concept of $k$-equivalence of permutation groups, where $k$ is a natural number. Namely,
two permutation groups $G$ and $H$ on a set $\Omega$ are said to be \emph{$k$-equivalent} if for every  $k$-ary relation
$s\subseteq\Omega^k$, we have $s^G=s$ if and only if $s^H=s$. In each class of $k$-equivalent groups there is a unique maximal (with respect
to inclusion) group, which is called the \emph{$k$-closure} of each group $G$ belonging to this class and is denoted by~$G^{(k)}$. Wielandt
proved \cite{Wielandt1969} (see also Proposition~\ref{basicProps} below) that
\begin{equation}\label{e:w}
\sym(\Omega)\ge G^{(1)}\ge G^{(2)}\ge\cdots \ge G^{(k)}=G^{(k+1)}=\cdots = G
\end{equation}
for some $k$. Thus, the closures of the group $G$ can be considered as approximations to~$G$.
On the other hand, the concept of the closure is a powerful tool for dealing with the above-mentioned
computational problem, because the $k$-closures are exactly the groups $\aut(\mathfrak S)$
where $\mathfrak S$ is the set of relations of arity at least~$k$ left invariant by the action of \( G \).

From the computational point of view, the $k$-closure problem consists in finding the $k$-closure of a given permutation group (it is
assumed that permutation groups are given by  generating sets, see \cite{Sere2002}). Note that the case $k=1$ is trivial, because the
$1$-closure of any permutation group~$G$ is the direct product of symmetric groups acting on the orbits of~$G$. In this setting,
polynomial-time algorithms for computing the $2$-closure were constructed for the nilpotent groups \cite{P1994}, the groups of odd
order~\cite{EP2001}, and the supersolvable groups~\cite{PV2020}. Very recently it was shown that the $3$-closure of a solvable group can also be 
computed in polynomial time~\cite{PV2024}.

The above computational results are based on some ``similarity'' of $k$-equivalent groups. It is clear from~\eqref{e:w} that the larger the
number $k$ is the more ``similar'' the $k$-equivalent groups are. In \cite{Wielandt1969}, Wielandt proved that if groups~$G$ and $H$ are
$k$-equivalent and $k\ge 1$ (respectively, $k\ge 2$, $k\ge m$), then $G$ is transitive (respectively, imprimitive, $m$-transitive) if and
only if so is~$H$. More interestingly, however, that $k$-equivalent groups are ``similar'' not only as permutation groups, but also as
abstract groups. For example, if $G$ is an abelian group (respectively, a $p$-group, a group of odd order), then, for $k\ge 2$, the
group~$H$ is an abelian group (respectively, a $p$-group, a group of odd order), see \cite{Wielandt1969}. Recently, it was
proved~\cite{BPVV2021} that a similar statement is true for solvable groups if $k\ge 3$ (the example of $2$-transitive solvable groups shows
that $k=2$ cannot be taken here). Each of these statements expresses the fact that the $k$-closure (for the corresponding $k$) of any
permutation group of the corresponding class also lies in it. This fact, a polynomial upper bound on the orders of solvable primitive groups
\cite{Pal1982}, and the Babai--Luks algorithm \cite{BL1983} form a foundation of the computational results from the previous paragraph
(except for the case of supersolvable groups, where the algorithm is constructed on more subtle arguments, because the $2$-closure of a
supersolvable group need not be even solvable).

In fact, the Babai--Luks algorithm \cite{BL1983} can be applied in a much broader situation, for the class of $\alt(d)$-free groups. Recall that $H$ is a {\em section} of a group $G$ if $H$ is a homomorphic image of a subgroup of $G$. Now a group is {\em$\alt(d)$-free}, $d\geq5$, if it does not contain a section isomorphic to the alternating group of degree~$d$. In particular, solvable groups are $\alt(5)$-free. Furthermore, the classical result by Babai, Cameron, and P{\'a}lfy \cite{BCP1982} states, up to the language (see, e.g., a remark before Theorem III in \cite{PS1997}), that the size of an $\alt(d)$-free primitive permutation group of degree $n$ is bounded from above by $n^c$, where $c$ is a constant depending only on~$d$. Thus, in the context of the previous paragraph, we arrive at the question when the $k$-closure of an $\alt(d)$-free group is also $\alt(d)$-free. The goal of the present paper is to answer this question even for a wider class of groups.

To proceed, we recall another concept introduced by Wielandt \cite[Definition~11.3]{Wie1964}. A class $\X$ of (abstract) groups is said to
be {\em  complete} if it is closed with respect to taking subgroups, quotients, and extensions. The examples are the class of all groups,
$p$-groups for any fixed prime $p$, and solvable groups. In fact, given a fixed set $S$ (finite or infinite) of simple groups, the class of
groups whose composition factors are sections of groups from $S$ is complete. 

We are ready to state our main result.

\begin{theorem}\label{t:main}
	Let $\X$ be a complete class including all $\alt(25)$-free groups.  Then the $k$-closure of every permutation
	group from $\X$ belongs to $\X$ for each $k\geq4$.
\end{theorem}

Since the class of all $\alt(d)$-free groups is complete, the statement below is an immediate consequence of Theorem~\ref{t:main}.

\begin{corollary}\label{c:main}
	If $G$ is an $\alt(d)$-free group with $d\geq25$, then $G^{(k)}$ is $\alt(d)$-free group for $k\geq4$.
\end{corollary}

Note that the constant $4$ in Theorem~\ref{t:main} is the best possible and, if $k=4$, then the same holds true for the constant $25$.
This follows from the two examples below.%
\medskip

(i) The affine group $G=\mathrm{AGL}_m(2)$ is $3$-transitive in its natural action on a linear space of dimension $m\geq2$ over the
field of order~$2$.  It follows that \( G^{(3)} = \sym(2^m) \), so the theorem does not hold for $k\leq3$.
\smallskip

(ii) The Mathieu group \( G = M_{24} \) is $\alt(9)$-free and acts on 24 points 5-transitively, so \( G^{(4)} = \sym(24) \) is not
$\alt(24)$-free.
\medskip

It should be also noted that there are a lot of complete classes of groups distinct from the class of $\alt(d)$-free groups, where
$d\geq25$, that satisfy the hypothesis of Theorem~\ref{t:main}. Indeed, consider any nonabelian simple group $S$ containing a section
isomorphic to $\alt(d)$. Then the class of all groups whose nonabelian factors are either $\alt(25)$-free or sections of $S$ is such an example.

Let us now briefly discuss the strategy of the proof of the main result. Suppose that $\X$ is a complete class of groups satisfying the
hypothesis of Theorem~\ref{t:main} and ${k\geq4}$. We are interested in the question of when $G^{(k)} \in \X$ for any
permutation group $G\in \X$. Note that along with each pair of permutation groups, the class $\X$ contains their transitive and intransitive
direct products, as well as their imprimitive and primitive wreath products. Explicit formulas for the $k$-closures of
these products are known, see \cite{KK,PV2020}, hence our question is reduced to the case when the group $G$ is
primitive and even basic \cite{PV2024}. Recall that a permutation group is  \emph{nonbasic} if it is contained in a wreath product in the
product action; it is \emph{basic} otherwise, see~\cite[Section~4.3]{CB}.

Apparently, the first nontrivial results on the closures of primitive groups were obtained in \cite{LPS88,PS1992}, where it was proved that
in most cases the socle of the group is preserved when taking the $k$-closure, and all the cases were described when this is not so. These
results allow us to reduce the situation to the case of  the affine groups (see Section~\ref{secProof}). When working with
affine groups, it is natural to use Aschbacher’s classification \cite{A1984}, which divides them structurally into $9$ classes. An important
step in the study of $k$-closures of affine groups was made in~\cite{Xu2006,XuGLP2011}, where for each Aschbacher class an explicit value
for the number $k$ was found such that the $k$-closure of each group from this class remains in the same class. Following the Aschbacher
classification, we consider the classes to which group $G$ might belong. While excluding the Aschbacher classes related with
tensor decompositions of the underlying linear space, we need some new tools to control the
$k$-equivalence and $k$-closure for stabilizers of the tensor decompositions, which are the subject of Section~\ref{secTensor}. Finally
(Subsections~\ref{ss:c6} and~\ref{ss:c9}), we exclude the cases when the group  $G$ is  of symplectic type  or belongs to the last
Aschbacher class (the classes \( \C_6 \) and \( \C_9 \) in the notation in~\cite{XuGLP2011}), thus completing the proof of
the theorem.

To make the paper as self-contained as possible, we cite the relevant results concerning the $k$-equivalence and $k$-closure of permutation
groups in Section~\ref{secPre}. The concluding remarks are collected at the end of the paper (Section~\ref{s:end}).

\section{Preliminaries}\label{secPre}

Let \( G \) be a permutation group on \( \Omega \), and let \( \krel(G) \), \( k \geq 1 \), denote
the set of \( G \)-invariant \( k \)-ary relations.

\begin{lemma}{\em\cite[Theorem~4.7]{Wielandt1969}}\label{krellem}
	Let \( G \) and \( H \) be permutation groups on the same set \( \Omega \),
	and \( k \geq 1 \). Then \( \krel(H) \subseteq \krel(G) \) if and only if for any
	\( x_1, \dots, x_k \in \Omega \) and any \( g \in G \) there exists some \( h \in H \)
	such that \( x_i^g = x_i^h \) for all \( i = 1, \dots, k \).
\end{lemma}

We say that \( G \) is \emph{\( k \)-equivalent} to \( H \) if \( \krel(G) = \krel(H) \).
Observe that the \( k \)-closure of \( G \) is the largest permutation group on \( \Omega \) which is \( k \)-equivalent to \( G \).
Recall that a \emph{base} of a permutation group is a subset whose pointwise stabilizer is trivial.
We collect some basic properties of \( k \)-closures and \( k \)-equivalence which will be used all throughout the text.

\begin{proposition}\label{basicProps}
	Let \( G \) and \( H \) be permutation groups on \( \Omega \). 
	\begin{enumerate}[{\rm(i)}]
		\item For \( k \geq 1 \), if \( G \) and \( H \) are \( (k+1) \)-equivalent,
			then they are \( k \)-equivalent. In particular, \( G^{(k+1)} \leq G^{(k)} \).
		\item For \( k \geq 2 \), if \( G \) and \( H \) are \( k \)-equivalent, then their orbits and systems of imprimitivity are the same.
		\item For \( k \geq 2 \) and \( \alpha \in \Omega \), if \( G \) and \( H \) are \( k \)-equivalent,
			then \( G_\alpha \) is \( (k-1) \)-equivalent to \( H_\alpha \).
		\item For \( k \geq 1 \), if \( G \leq H \), then \( G^{(k)} \leq H^{(k)} \).
		\item For \( k \geq 1 \), we have \( (G^{(k)})^{(k)} = G^{(k)} \).
		\item Suppose \( G \) has a base of size \( k \geq 1 \), then \( G^{(k+1)} = G \).
			In particular, if \( G \) has a faithful regular orbit, then \( G^{(2)} = G \).
	\end{enumerate}
\end{proposition}
\begin{proof}
	Parts~(i)-(vi) follow from~ Theorems~4.3, 4.11, 4.12, 5.7, 5.9, and~5.12 of \cite{Wielandt1969}, respectively.
\end{proof}

For \( \Delta \subseteq \Omega \) let \( G^\Delta \) denote the group induced on \( \Delta \) by the setwise stabilizer of \( \Delta \) in \( G \).
We may view \( G^\Delta \) as a subgroup of \( \sym(\Delta) \).

Let \( \Delta \) be some block of imprimitivity for a transitive permutation group \( G \). Translates of \( \Delta \)
give a system of imprimitivity of \( G \) which we will denote by \( \Omega/\Delta \). Let \( G^{\Omega/\Delta} \) denote
the permutation group induced by \( G \) on \( \Omega/\Delta \) by permuting the blocks. Clearly, \( G^\Delta \) and \( G^{\Omega/\Delta} \)
are sections of \( G \), so if \( G \) lies in some complete class~\( \X \) then \( G^\Delta \in \X \) and \( G^{\Omega/\Delta} \in \X \).
Moreover, \( G \) can be embedded into the \emph{wreath product in imprimitive action} \( G^\Delta \wr G^{\Omega/\Delta} \).

Given two permutation groups \( L \leq \sym(\Delta) \) and \( K \leq \sym(\Gamma) \), let \( L \uparrow K \leq \sym(\Delta^\Gamma) \) denote
the \emph{wreath product in product action}. As an abstract group it is isomorphic to \( L^{|\Gamma|} \rtimes K \), where
\( L^{|\Gamma|} \) acts on \( \Delta^\Gamma \) coordinatewise, while \( K \) permutes the coordinates. A primitive group \( G \) is called \emph{basic}
if it cannot be embedded into the wreath product in product action for \( |\Delta| > 1 \), \( |\Gamma| > 1 \).

Let \( G \leq \sym(\Omega) \) be a primitive permutation group preserving a nontrivial product decomposition \( \Omega = \Delta^\Gamma \).
Without loss of generality we may assume that \( \Gamma = \{ 1, \dots, m \} \) for some \( m > 1 \). For \( \gamma \in \Gamma \) define a partition of \( \Omega \) into \( |\Delta| \) parts
\[ \Delta_\gamma = \{ \Delta \times \dots \times \Delta \times \{ x \} \times \Delta \times \dots \times \Delta \mid x \in \Delta \}, \]
i.e.\ each part of \( \Delta_\gamma \) consists of those tuples from \( \Delta^\Gamma \) which have its \( \gamma \)-th coordinate constant.
There is a natural bijective correspondence between the partitions \( \Delta_\gamma \), \( \gamma \in \Gamma \), and the points of \( \Gamma \).
The group \( G \) permutes the partitions \( \Delta_\gamma \), \( \gamma \in \Gamma \), so let \( G^\Gamma \) denote the permutation group on \( \Gamma \)
obtained through this correspondence. With some abuse of notation, let \( G^{\Delta_\gamma} \) denote the permutation group on \( \Delta \) induced
by the stabilizer of \( \Delta_\gamma \) in~\( G \). In the terminology of~\cite[Section~5.3]{PSn2018}, the group \( G^{\Delta_\gamma} \) is the \emph{\( \gamma \)-component} of \( G \),
and there is an explicit formula for it~\cite[(5.11)]{PSn2018}.

Clearly, \( G^\Gamma \) and \( G^{\Delta_\gamma} \) are sections of \( G \), and~\cite[Theorem~5.14~(ii)]{PSn2018}
implies that \( G \) can be embedded into the wreath product \( G^{\Delta_\gamma} \uparrow G^\Gamma \) in product action.

To compute the \( k \)-closure of a group in product action we need another closure operator. Observe that a permutation group \( K \leq \sym(\Gamma) \)
acts on the set of ordered partitions of \( \Gamma \) in at most \( r \geq 1 \) classes. Let \( K^{[r]} \) denote the largest permutation group
on \( \Gamma \) having the same orbits on the set of ordered partitions into at most \( r \) parts as \( K \). This new operator, which we call
the \emph{\( r \)-closure with respect to partitions}, is indeed a closure operator and enjoys some properties similar to the \( k \)-closure, see~\cite[Section~3]{PV2021}.
In particular, \( K^{[r+1]} \geq K^{[r]} \) and \( K^{[m]} = K \) for \( m = |\Gamma| \).

There is a direct connection between these two types of closures.

\begin{proposition}[{\cite[Lemma~3.2]{PV2021}}]\label{partClkCl}
	Let \( K \leq \sym(\Gamma) \) be a permutation group. For \( k \geq 1 \) we have \( K^{[k+1]} \leq K^{(k)} \).
\end{proposition}

Let \( \mathrm{Orb}_k(K) \) denote the set of orbits of \( K \leq \sym(\Gamma) \) on \( \Gamma^k \).

\begin{proposition}\label{basicEmb}
	Let \( G \) and \( H \) be permutation groups on \( \Omega \). 
	\begin{enumerate}[{\rm(i)}]
		\item Let \( \Delta \subset \Omega \) be an orbit of \( G \). For any \( k \geq 1 \) we have \( (G^{(k)})^\Delta \leq (G^\Delta)^{(k)} \).
			Moreover, if \( H \) is \( k \)-equivalent to \( G \) then \( G^\Delta \) is \( k \)-equivalent to \( H^\Delta \).
		\item Let \( G \) be a transitive imprimitive group with a nontrivial block \( \Delta \subset \Omega \).
			For any \( k \geq 2 \),
			\[ (G^\Delta \wr G^{\Omega/\Delta})^{(k)} = (G^\Delta)^{(k)} \wr (G^{\Omega/\Delta})^{(k)}.  \]
			Moreover, if \( H \) is \( k \)-equivalent to \( G \), then \( H \) is also imprimitive, and \( G^\Delta \) is \( k \)-equivalent to \( H^\Delta \)
			and \( G^{\Omega/\Delta} \) is \( k \)-equivalent to \( H^{\Omega/\Delta} \).
		\item Let \( G \) be a primitive nonbasic group preserving a nontrivial product decomposition \( \Omega = \Delta^\Gamma \),
			and let \( \gamma \in \Gamma \). For any \( k \geq 2 \) we have
			\[ (G^{\Delta_\gamma} \uparrow G^\Gamma)^{(k)} = (G^{\Delta_\gamma})^{(k)} \uparrow (G^\Gamma)^{[r]}, \]
			where \( r = \min \{ |\mathrm{Orb}_k(G^{\Delta_\gamma})|,\, |\Gamma| \} \).
			Moreover, if \( H \) is \( k \)-equivalent to \( G \), then \( H \) is also nonbasic, and \( G^{\Delta_\gamma} \) is \( k \)-equivalent to \( H^{\Delta_\gamma} \)
			and \( (G^\Gamma)^{[r]} = (H^\Gamma)^{[r]} \).
	\end{enumerate}
\end{proposition}
\begin{proof}
	(i) Since \( \krel(G^\Delta) = \krel(G) \cap \Delta^k \), groups \( G^\Delta \) and \( H^\Delta \) are \( k \)-equivalent and the second part of the claim follows.
	Now take \( H = G^{(k)} \) and observe that \( (G^{(k)})^\Delta \) is \( k \)-equivalent to \( G^\Delta \), hence \( (G^{(k)})^\Delta \leq (G^\Delta)^{(k)} \).

	(ii) Follows from~\cite[Theorem~2.4 and Lemma~2.5]{KK}.

	(iii) The closure formula is the main result of~\cite{PV2021}. It follows from the closure formula that \( H \leq (G^{\Delta_\gamma})^{(k)} \uparrow (G^\Gamma)^{[r]} \)
	so \( H \) is nonbasic. It also follows that \( H^{\Delta_\gamma} \leq (G^{\Delta_\gamma})^{(k)} \) and \( H^\Gamma \leq (G^\Gamma)^{[r]} \).
	Therefore \( (H^{\Delta_\gamma})^{(k)} \leq (G^{\Delta_\gamma})^{(k)} \) and \( (H^\Gamma)^{[r]} \leq (G^\Gamma)^{[r]} \).
	By exchanging the roles of \( H \) and \( G \) we derive \( (G^{\Delta_\gamma})^{(k)} \leq (H^{\Delta_\gamma})^{(k)} \)
	and \( (G^\Gamma)^{[r]} \leq (H^\Gamma)^{[r]} \). Hence \( (G^{\Delta_\gamma})^{(k)} = (H^{\Delta_\gamma})^{(k)} \)
	which implies that \( G^{\Delta_\gamma} \) is \( k \)-equivalent to \( H^{\Delta_\gamma} \), and \( (G^\Gamma)^{[r]} = (H^\Gamma)^{[r]} \) as claimed.
\end{proof}

Note that in Proposition~\ref{basicEmb}~(iii) when \( k \geq 3 \) we have \( |\mathrm{Orb}_k(G^{\Delta_\gamma})| \geq k+1 \) since \( \Delta \geq 2 \), so Proposition~\ref{partClkCl}
implies that \( (G^{\Delta_\gamma} \uparrow G^\Gamma)^{(k)} \leq (G^{\Delta_\gamma})^{(k)} \uparrow (G^\Gamma)^{(k)} \), cf.~\cite[Theorem~1.2]{PV2021}.

Recall that a permutation group \( G \) on \( \Omega \) is called \emph{affine} if it contains a normal elementary abelian subgroup \( V \)
acting regularly on \( \Omega \). In that case \( \Omega \) can be identified with \( V \) and \( G \) decomposes into a semidirect product \( G = V \rtimes G_0 \),
where \( G_0 \) is the zero stabilizer. The group \( G_0 \) acting on \( V \) can be viewed as a subgroup of \( \GL(V) \),
and \( G \) is primitive if and only if \( G_0 \) acts irreducibly on \( V \). In this case, \( V = \soc(G) \).

\begin{lemma}\label{affcl}
	Let \( V \) be a vector space over a prime field, and let \( G = V \rtimes G_0 \), \( G_0 \leq \GL(V) \),
	be an affine group with socle \( V \). Suppose that \( G^{(k)} \), \( k \geq 2 \), is also an affine group with socle \( V \),
	i.e.\ \( G^{(k)} = V \rtimes H_0 \) for some \( H_0 \leq \GL(V) \).
	Then \( H_0 = G_0^{(k-1)} \cap \GL(V) \).
\end{lemma}
\begin{proof}
	By Proposition~\ref{basicProps}~(iii), \( G_0 \) is \( (k-1) \)-equivalent to \( H_0 \), hence \( H_0 \leq G_0^{(k-1)} \cap \GL(V) \).
	To prove the converse inclusion, let \( h \in G_0^{(k-1)} \cap \GL(V) \) be arbitrary. Let \( v_1, \dots, v_k \in V \)
	be arbitrary vectors. Since \( h \in G_0^{(k-1)} \), by Lemma~\ref{krellem} there exists \( g \in G_0 \) such that \( (v_i - v_1)^g = (v_i - v_1)^h \)
	for \( i = 2, \dots, k \). As elements \( h \) and \( g \) lie in \( \GL(V) \), we have \( v_i^h = v_i^g + v_1^h - v_1^g \)
	for all \( i = 2, \dots, k \). There exists some element \( t \in V \leq G \) acting on \( V \) by translation by
	the vector \( v_1^h - v_1^g \), that is, \( v^t = v + v_1^h - v_1^g \) for all \( v \in V \). Therefore
	\( v_i^h = v_i^{gt} \) for all \( i = 2, \dots, k \). Notice that this equality holds for \( i = 1 \) as well,
	indeed, \( v_1^{gt} = v_1^g + v_1^h - v_1^g = v_1^h \). Finally, \( gt \) lies in \( G \), so \( h \) lies in \( G^{(k)} \) by Lemma~\ref{krellem},
	as the choice of \( v_1, \dots, v_k \in V \) was arbitrary.
	Hence \( G_0^{(k-1)} \cap \GL(V) \leq G^{(k)} \) and since \( G_0^{(k-1)} \cap \GL(V) \) stabilizes the zero vector,
	we derive \( G_0^{(k-1)} \cap \GL(V) \leq H_0 \). 
\end{proof}

The \( k \)-closure of a primitive affine group is also an affine group with the same socle for \( k \geq 4 \) according to the following result.

\begin{lemma}{\em\cite[Lemma~4.1]{PS1992}}\label{4claff}
	Let \( G \) be a primitive affine permutation group with socle \( V \), and assume \( k \geq 4 \).
	Then \( G^{(k)} \) is also an affine group with socle~\( V \).
\end{lemma}

The following proposition implies that the groups of Lie type of bounded rank are \( \alt(d) \)-free for \( d \) bounded in terms of the rank.
A less precise inequality but with a more elementary proof can also be found in~\cite[Theorem~5.7A]{DM}.

\begin{proposition}\label{altGL}
	Suppose that \( \alt(d) \) is a section of \( \GL_a(q) \) for some \( a \geq 1 \), \( q \geq 2 \).
	Then \( a \geq d-2 \) for \( d \geq 9 \).
\end{proposition}
\begin{proof}
	Assume that \( a \) is the minimal integer such that \( \GL_a(\F) \) contains a section isomorphic to \( \alt(d) \), \( d \geq 9 \),
	for some field \( \F \) of positive characteristic.
	By~\cite[Lemma~5.7D]{DM}, for some algebraically closed field \( \mathbb{K} \) of characteristic \( p > 0 \) there exists a subgroup \( G \) of \( \GL_a(\mathbb{K}) \)
	such that \( G/Z(G) \simeq \alt(d) \), and \( Z(G) \) is a group of scalars. In particular, \( \alt(d) \) has a faithful projective \( p \)-modular
	representation of degree~\( a \). By~\cite[Proposition~5.3.7~(i)]{KL} we have \( a \geq d-2 \).
\end{proof}

When dealing with \( k \)-closures of affine permutation groups, we will use Aschbacher's classification of subgroups of linear groups~\cite{A1984}.
Xu, Giudici, Li and Praeger~\cite{XuGLP2011} proved that Aschbacher's classes in \( \GaL_a(q) \) are preserved by \( k \)-closures
for suitably chosen \( k \), depending on the class; here we view a subgroup of \( \GaL_a(q) \) as a permutation group on \( \F_q^a \).

We will follow~\cite{XuGLP2011} in our definitions and notations of Aschbacher classes. Recall that a subgroup of \( \GaL_a(q) \) not containing
\( \mathrm{SL}_a(q) \) lies in one of the nine classes \( \C_1, \dots, \C_9 \), defined mostly in terms of some geometric structure preserved by the groups lying in the class.
We briefly outline the classes and relevant structures:
\smallskip

\( \C_1 \): Groups preserving a nontrivial proper subspace.

\( \C_2 \): Groups acting imprimitively on the vector space.

\( \C_3 \): Groups preserving the structure of an extension field.

\( \C_4 \): Groups preserving a nontrivial decomposition of the vector space into the tensor product of two spaces of unequal dimensions.

\( \C_5 \): Groups preserving the structure of a proper subfield.

\( \C_6 \): Groups normalizing a subgroup of symplectic type.

\( \C_7 \): Groups preserving a nontrivial decomposition of the vector space into the tensor product of several spaces of equal dimensions.

\( \C_8 \): Groups preserving a nondegenerate alternating, hermitian or quadratic form. We will use notation \( \C_{\mathrm{Sp}} \), \( \C_{\mathrm U} \) and \( \C_{\mathrm O} \),
respectively, depending on the form preserved.

\( \C_9 \): Groups which are not contained in any of \( \C_1, \dots, \C_8 \). These groups are almost simple modulo center.
\smallskip

Strictly speaking, each class \( \C_i \), \( i = 1, \dots, 8 \), is defined by its maximal members.
For example, \( G \leq \GaL_a(q) \) lies in \( \C_4 \) if \( V = \F_q^a \) can be decomposed as \( V = U \otimes W \),
\( \dim U > \dim W > 1 \), and \( G \) can be conjugated inside the central product \( (\GL(U) \otimes \GL(W)) \rtimes \aut(\F_q) \) acting naturally on \( V \).
By~\cite[Proposition~4.4.1]{XuGLP2011}, this is equivalent to saying that \( G \) preserves the set of simple tensors, i.e.\ a unary relation
of the form \( \{ u \otimes w \mid u \in U,\, w \in W \} \). This implies that \( G^{(1)} \cap \GaL_a(q) \) also lies in the class \( \C_4 \).
Essentially~\cite[Theorem~1.1]{XuGLP2011} says that one can find similar relations for all of Aschbacher classes.
Here we give a summary of the main results of~\cite{XuGLP2011}.

\begin{proposition}\label{xukcl}
	Let \( G \leq \GaL_a(q) \), \( a \geq 2 \), act on \( \F_q^a \) naturally.
	\begin{enumerate}[{\rm(i)}]
		\item If \( G \in \C_i \) for \( i \in \{1, \dots, 7, \mathrm{Sp}, \mathrm{U}, \mathrm{O} \} \), then \( G^{(k)} \cap \GaL_a(q) \in \C_i \)
			for \( k = 2 \) if \( i \in \{ 3, 6, \mathrm{Sp} \} \) and \( k = 1 \) otherwise.
		\item If \( G \in \C_9 \), then \( G^{(2)} \cap \GaL_a(q) \in \C_9 \), unless \( a = 4 \), \( q = 2 \) and \( G = \alt(7) \).
	\end{enumerate}
\end{proposition}
\begin{proof}
	Part~(i) follows from~\cite[Theorem~1.1]{XuGLP2011}, and part~(ii) from~\cite[Proposition~3.3.1]{XuGLP2011}.
\end{proof}

Note that Aschbacher's theorem~\cite{A1984} applies not only to \( \GaL_a(q) \) but to the automorphism groups of the classical groups in general.
Classes \( \C_1, \dots, \C_9 \) are defined similarly for the classical groups, and if a group \( G \) belongs to some \( \C_i \), \( i = 1, \dots, 8 \),
defined for a classical group, then \( G \) lies in the same class but defined for \( \GaL_a(q) \), if one forgets about the classical geometry preserved by \( G \).

\section{Closures of tensor products}\label{secTensor}

In this section \( \X \) is an arbitrary complete class which contains all solvable groups.
In our descriptions of subgroups stabilizing tensor decompositions we follow~\cite[Sections~4.4 and~4.5]{XuGLP2011}.

\subsection{\texorpdfstring{Stabilizers of \( X \otimes Y \)}{Stabilizers of 2-fold tensor product}}

Let \( V = X \otimes Y \) be a tensor product of vector spaces over a finite field~\( \F \).
We will also use the symbol \( \otimes \) to denote the central product of two linear groups acting naturally on the corresponding tensor product of vector spaces.
Recall that the central product \( \GL(X) \otimes \GL(Y) \) acts on \( V \) by the following rule, if \( g_X \in \GL(X) \),
\( g_Y \in \GL(Y) \) and \( x \in X \), \( y \in Y \), then \( (x \otimes y)^{g_X \otimes g_Y} = x^{g_X} \otimes y^{g_Y} \).

The group of field automorphisms \( \aut(\F) \) acts on \( V \) as follows.
First, choose an \( \F \)-basis \( x_1, \dots, x_r \) of \( X \). Every \( x \in X \) can be decomposed
uniquely into a sum \( x = \alpha_1 x_1 + \dots + \alpha_r x_r \), and one can define the action of \( \sigma \in \aut(\F) \) on \( X \)
by \( x^\sigma = \alpha_1^\sigma x_1 + \dots + \alpha_r^\sigma x_r \). By choosing some basis of \( Y \) one defines the
action of \( \aut(\F) \) in the similar way. Finally, the action of \( \sigma \in \aut(\F) \) on \( V = X \otimes Y \) is defined
by \( (x \otimes y)^\sigma = x^\sigma \otimes y^\sigma \), where \( x \in X \), \( y \in Y \).

It follows that the group \( (\GL(X) \otimes \GL(Y)) \rtimes \aut(\F) \) preserves the tensor decomposition \( V = X \otimes Y \),
and when \( \dim X \neq \dim Y \) it is the largest group preserving that decomposition, see~\cite[Formula~(4.4.1)]{XuGLP2011}.
Every element \( g \) of that group can be written in the form
\( g = (g_X \otimes g_Y)\sigma \), where \( g_X \in \GL(X) \), \( g_Y \in \GL(Y) \) and \( \sigma \in \aut(\F) \).
If \( g \) admits a second decomposition \( g = (g_X' \otimes g_Y')\sigma' \), then \( \sigma' = \sigma \) and there
exists some nonzero \( \alpha \in \F \) such that \( g_X' = \alpha \cdot g_X \) and \( g_Y' = \frac{1}{\alpha} \cdot g_Y \).

Take \( G \leq (\GL(X) \otimes \GL(Y)) \rtimes \aut(\F) \) where \( X \) and \( Y \) are vector spaces over~\( \F \).
Recall that we can identify \( \F^\times \) with the center of \( \GL(X) \) and \( \GL(Y) \).
For \( g = (g_X \otimes g_Y)\sigma \in G \) define the cosets \( \pi_X(g) = \F^\times \cdot g_X\sigma \) and \( \pi_Y(g) = \F^\times \cdot g_Y\sigma \).
Note that these elements are defined correctly. Indeed, if \( g = (g_X' \otimes g_Y')\sigma' \), then \( \sigma' = \sigma \)
and \( g_X' = \alpha \cdot g_X \) for some \( \alpha \in \F^\times \). Hence \( \F^\times \cdot g_X'\sigma' = \F^\times \cdot \alpha g_X\sigma
= \F^\times \cdot g_X\sigma \), so \( \pi_X(g) \) does not depend on the choice of the decomposition of \( g \). Similar reasoning holds
for \( \pi_Y \), so \( \pi_X : G \to \PGaL(X) \) and \( \pi_Y : G \to \PGaL(Y) \) are correctly defined maps,
and moreover, \( \pi_X \) and \( \pi_Y \) are homomorphisms.

Assume that \( \F^\times = Z(\GL(V)) \leq G \), so \( \ker \pi_X \cap \ker \pi_Y = \F^\times \).
Since \( \pi_X(G) \leq \PGaL(X) \), we can consider the full preimage \( G_X \) of \( \pi_X(G) \) in \( \GL(X) \);
define \( G_Y \) accordingly. By construction, \( G_X \) and \( G_Y \) contain \( \F^\times \).
Now, by Remak's theorem, \( G/\F^\times \) embeds into a direct product
\( \pi_X(G) \times \pi_Y(G) \leq \PGaL(X) \times \PGaL(Y) \), so \( G/\F^\times \) embeds into \( G_X/\F^\times \times G_Y/\F^\times \).

\begin{proposition}\label{twotens}
	Let \( X \) and \( Y \) be vector spaces over a finite field \( \F \), and let \( V = X \otimes Y \).
	Let \( G \) and \( H \) be subgroups of \( (\GL(X) \otimes \GL(Y)) \rtimes \aut(\F) \leq \GaL(V) \).
	Assume that \( \F^\times \leq G \) and \( \F^\times \leq H \), and \( G \) and \( H \) are \( k \)-equivalent
	for some \( k \geq 1 \). Then \( G_X \) is \( k \)-equivalent to \( H_X \), and \( G_Y \) is \( k \)-equivalent to \( H_Y \).
\end{proposition}
\begin{proof}
	We will show that \( \krel(H_X) \subseteq \krel(G_X) \). By Lemma~\ref{krellem}, it is sufficient to show
	that for all \( x_1, \dots, x_k \in X \) and \( \gamma_X \in G_X \) there exists
	\( \beta_X \in H_X \) such that \( x_i^{\gamma_X} = x_i^{\beta_X} \), \( i = 1, \dots, k \).
	Notice that we may assume that the vector \( x_i \) is nonzero, \( i = 1, \dots, k \).

	Let \( y \in Y \) be an arbitrary nonzero vector. Since \( \gamma_X \) lies in \( G_X \),
	there exists \( g \in G \) such that \( g = (g_X \otimes g_Y)\sigma \), where
	\( g_X \in \GL(X) \), \( g_Y \in \GL(Y) \), \( \sigma \in \aut(\F) \) and \( \gamma_X = g_X\sigma \).
	Let \( \gamma_Y = g_Y\sigma \). As \( G \) is \( k \)-equivalent to \( H \), there exists an element \( h \in H \) such that
	\( (x_i \otimes y)^g = (x_i \otimes y)^h \), \( i = 1, \dots, k \).

	Write \( h = (h_X \otimes h_Y)\sigma' \), where \( h_X \in \GL(X) \), \( h_Y \in \GL(Y) \), \( \sigma' \in \aut(\F) \).
	Let \( \gamma_X' = h_X\sigma' \) and \( \gamma_Y' = h_Y\sigma' \). We have
	\[ x_i^{\gamma_X} \otimes y^{\gamma_Y} = (x_i \otimes y)^g = (x_i \otimes y)^h = x_i^{\gamma_X'} \otimes y^{\gamma_Y'}, \]
	for \( i = 1, \dots, k \). It follows that there exist \( \alpha_i \in \F^\times \) such that
	\( x_i^{\gamma_X} = \alpha_i \cdot x_i^{\gamma_X'} \) and \( y^{\gamma_Y} = \frac{1}{\alpha_i} \cdot y^{\gamma_Y'} \)
	for \( i = 1, \dots, k \). As \( y \) is nonzero, \( \alpha_i \) does not depend on \( i \), i.e.\ \( \alpha_i = \alpha \)
	for \( i = 1, \dots, k \).

	Since \( \F^\times \leq H_X \), there exists \( z \in H_X \) such that
	\( x^z = \alpha \cdot x \) for all \( x \in X \). Therefore \( x_i^{\gamma_X} = x_i^{\gamma_X' z} \)
	for \( i = 1, \dots, k \). Now \( \beta_X = \gamma_X' \cdot z \in H_X \) is the required element.

	By switching the roles of \( G \) and \( H \), we obtain \( \krel(G_X) \subseteq \krel(H_X) \).
	Thus \( \krel(H_X) = \krel(G_X) \) and \( G_X \) is \( k \)-equivalent to \( H_X \), as claimed.
	Similarly, \( \krel(H_Y) = \krel(G_Y) \) and \( G_Y \) is \( k \)-equivalent to \( H_Y \).
\end{proof}

\begin{corollary}\label{clTens}
	Let \( V \) be a vector space over a finite field \( \F \), and let \( G \) be a primitive affine group with socle \( V \).
	Set \( G = V \rtimes G_0 \), where \( G_0 \) is the zero stabilizer, and assume that \( \F^\times \leq G_0 \).
	Suppose that \( G_0 \) stabilizes a nontrivial tensor decomposition \( V = X \otimes Y \)
	over \( \F \), where \( \dim X \neq \dim Y \). Finally, assume that \( k \geq 4 \) and the \( k \)-closures of
	\( X \rtimes (G_0)_X \) and \( Y \rtimes (G_0)_Y \) lie in the class \( \X \). Then \( G^{(k)} \) lies in the class~\( \X \).
\end{corollary}
\begin{proof}
	By Lemma~\ref{4claff}, the \( k \)-closures of \( G \), \( X \rtimes (G_0)_X \) and \( Y \rtimes (G_0)_Y \)
	are affine groups with socles \( V \), \( X \) and \( Y \), respectively.
	Set \( H = G^{(k)} \) and \( H = V \rtimes H_0 \), so \( G_0 \) is \( (k-1) \)-equivalent to \( H_0 \) by Proposition~\ref{basicProps}~(iii).
	By Proposition~\ref{xukcl}~(i), \( H_0 \) preserves some tensor decomposition of \( V \),
	and by~\cite[Lemma~4.4.5~(1)]{XuGLP2011} it preserves the decomposition \( V = X \otimes Y \).
	By Proposition~\ref{twotens}, \( (G_0)_X \) is \( (k-1) \)-equivalent to \( (H_0)_X \),
	and \( (G_0)_Y \) is \( (k-1) \)-equivalent to \( (H_0)_Y \). Therefore \( (H_0)_X \leq (G_0)_X^{(k-1)} \)
	and \( (H_0)_Y \leq (G_0)_Y^{(k-1)} \), hence
	\[ X \rtimes (H_0)_X \leq X \rtimes ((G_0)_X^{(k-1)} \cap \GL(X)) = (X \rtimes (G_0)_X)^{(k)}, \]
	where the last equality follows from Lemma~\ref{affcl}, and \( \GL(X) \) is the general linear group
	over the prime subfield of \( \F \). Similarly, \( Y \rtimes (H_0)_Y \leq (Y \rtimes (G_0)_Y)^{(k)} \).
	Therefore \( (H_0)_X \) and \( (H_0)_Y \) lie in \( \X \), and thus \( H \) lies in \( \X \) as claimed.
\end{proof}

\subsection{Subfield stabilizers}

Let \( V \) be a vector space over a finite field \( \F_q \) of order~\( q \), and let \( v_1, \dots, v_d \)
be the \( \F_q \)-basis of \( V \). Let \( \F_{q_0} \) be a proper subfield of \( \F_q \), and let \( V_0 \)
be the \( \F_{q_0} \)-span of \( v_1, \dots, v_d \). We say that \( G \leq \GaL(V) \) is a subfield subgroup,
if it preserves the set \( \F_qV_0 = \{ \lambda v \mid \lambda \in \F_q,\, v \in V_0 \} \) of lines passing through \( V_0 \).
The space \( V \) can be identified with a tensor product \( V_0 \otimes \F_q \) over \( \F_{q_0} \), and under such
an identification the group \( G \) preserves this tensor product, see~\cite[Formula (4.5.2)]{XuGLP2011}.
As in the previous section, let \( G_{V_0} \leq \GL(V_0) \) and \( G_{\F_q} \leq \GL_1(\F_q) \) denote the groups induced by \( G \) on \( V_0 \) and \( \F_q \) respectively.

\begin{proposition}\label{clSubf}
	Let \( V \) be a vector space over a finite field \( \F_q \) of order \( q \), and let \( G \)
	be a primitive affine group with socle \( V \). Set \( G = V \rtimes G_0 \), where \( G_0 \leq \GaL(V) \),
	and assume that \( \F_q^\times \leq G_0 \). Suppose that \( G_0 \) is a subfield subgroup, in particular,
	\( G_0 \) preserves the tensor decomposition \( V = V_0 \otimes \F_q \) over the field \( \F_{q_0} \).
	Assume that \( k \geq 4 \) and the \( k \)-closure of \( V_0 \rtimes (G_0)_{V_0} \) lies in the class \( \X \).
	Then \( G^{(k)} \) lies in the class \( \X \).
\end{proposition}
\begin{proof}
	By Lemma~\ref{4claff}, the \( k \)-closures of \( G \) and \( V_0 \rtimes (G_0)_{V_0} \) are affine groups with
	socles \( V \) and \( V_0 \), respectively. By Proposition~\ref{xukcl}~(i), \( G^{(k)} \) is also a subfield subgroup
	preserving \( V = V_0 \otimes \F_q \) over \( \F_{q_0} \). Set \( H = G^{(k)} \) and \( H = V \rtimes H_0 \),
	\( H_0 \leq \GaL(V) \). Since \( G_0 \) is \( (k-1) \)-equivalent to \( H_0 \), Proposition~\ref{twotens}
	implies that \( (H_0)_{V_0} \) is \( (k-1) \)-equivalent to \( (G_0)_{V_0} \). By Lemma~\ref{affcl}
	and our assumptions, \( (H_0)_{V_0} \) lies in \( \X \). Since \( (H_0)_{\F_q} \leq \GL_1(q) \),
	it follows that \( H_0 \) lies in class \( \X \), as required.
\end{proof}

\subsection{\texorpdfstring{Stabilizers of \( X \otimes \dots \otimes X \)}{Stabilizers of multifold tensor product}}

Let \( X \) be a vector space over a finite field \( \F \), and let \( V = X \otimes \dots \otimes X \) be the tensor
product of \( m \geq 2 \) copies of \( X \). An element \( g_1 \otimes \dots \otimes g_m \) of the central product
\( \GL(X) \otimes \dots \otimes \GL(X) \) acts on \( V \) naturally,
indeed, if \( x_1 \otimes \dots \otimes x_m \in V \), where \( x_i \in X \), \( i = 1, \dots, m \),
then \( (x_1 \otimes \dots \otimes x_m)^{g_1 \otimes \dots \otimes g_m} = x_1^{g_1} \otimes \dots \otimes x_m^{g_m} \).

A field automorphism \( \sigma \in \aut(\F) \) acts on \( V \) componentwise:
\[ (x_1 \otimes \dots \otimes x_m)^\sigma = x_1^\sigma \otimes \dots \otimes x_m^\sigma. \]
Finally, a permutation \( \tau \in \sym(m) \) acts on \( V \) by permuting components of simple tensors:
\[ (x_1 \otimes \dots \otimes x_m)^\tau = x_{\tau^{-1}(1)} \otimes \dots \otimes x_{\tau^{-1}(m)}, \]
and this action extends to the rest of \( V \) by linearity.
Since the action of field automorphisms commutes with the action of \( \sym(m) \),
we obtain the action of \( (\GL(X) \otimes \dots \otimes \GL(X)) \rtimes (\aut(\F) \times \sym(m)) \) on \( V \).
This is the stabilizer of the tensor decomposition \( V = X \otimes \dots \otimes X \), see~\cite[Formula~(4.4.3)]{XuGLP2011}.

Set \( L = (\GL(X) \otimes \dots \otimes \GL(X)) \rtimes (\aut(\F) \times \sym(m)) \) and as usual, identify
\( \F^\times \) with the center of \( \GL(X) \). For \( g \in L \) we can write
\( g = (g_1 \otimes \dots \otimes g_m)\sigma\tau \), where \( g_i \in \GL(X) \), \( i = 1, \dots, m \), \( \sigma \in \aut(\F) \)
and \( \tau \in \sym(m) \). For \( i \in \{ 1, \dots, m \} \) set \( \pi_i(g) = \F^\times \cdot g_i\sigma \),
and notice that it gives a correctly defined map \( \pi_i : L \to \PGaL(X) \).
The restriction of \( \pi_i \) to \( \GL(X) \otimes \dots \otimes \GL(X) \)
is a homomorphism into \( \PGaL(X) \) with kernel \( K_i = \GL(X) \otimes \dots \otimes \F^\times \otimes \dots \otimes \GL(X) \),
where \( \F^\times \) is on the \( i \)-th position in the central product.
We can also define a homomorphism \( \pi : L \to \sym(m) \) by the rule \( \pi(g) = \tau \).
The kernel of \( \pi \) is equal to \( (\GL(X) \otimes \dots \otimes \GL(X)) \rtimes \aut(\F) \).

Now define a map \( \Pi : L \to \PGaL(X) \wr \sym(m) \) by
\[ \Pi(g) = (\pi_1(g), \dots, \pi_m(g))\pi(g), \]
where \( g \in L \). Clearly \( \Pi \) is a homomorphism with kernel \( \bigcap_i K_i = \F^\times \).

We can define a faithful action of the abstract wreath product \( \PGaL(X) \wr \sym(m) \) on
\[ \mathcal{L} = \{ \F \cdot (v_1 \otimes \dots \otimes v_m) \mid v_i \in X,\, i = 1, \dots, m \} \]
which is the set of lines spanned by simple tensors. To do this, let \( \PGaL(X) \) act on the lines of \( X \),
and let \( \sym(m) \) permute the coordinates. The corresponding action is a product action,
so we may identify \( \PGaL(X) \wr \sym(m) \) with a wreath product in the product action \( \PGaL(X) \uparrow \sym(m) \leq \sym(\mathcal{L}) \),
The group \( L \) also acts on \( \mathcal{L} \) naturally, and one can easily see that \( g \in L \) and \( \Pi(g) \) act on \( \mathcal{L} \) in the same way,
i.e.\ they represent the same permutation from \( \sym(\mathcal{L}) \).

Let \( G \leq L \) be a subgroup containing \( \F^\times \), and assume that \( \pi(G) \) is a transitive
subgroup of \( \sym(m) \). Let \( G_X \leq \GaL(X) \) be the projection of \( G \) into the first component
of the tensor product, i.e.
\[ G_X = \{ g_1\sigma \mid (g_1 \otimes \dots \otimes g_m)\sigma\tau \in G,
\text{ where } \sigma \in \aut(\F),\, \tau \in \sym(m),\, \tau(1) = 1 \}. \]
Since \( \pi(G) \) is transitive, the projection \( G_X \) does not depend on the labelling of points \( \{ 1, \dots, m \} \)
up to conjugation in \( \GaL(X) \). One can similarly define the projection of \( \Pi(G) \leq \PGaL(X) \uparrow \sym(m) \)
into \( \PGaL(X) \); notice that this projection is \( G_X/\F^\times \).

By~\cite[Section~5.3]{PSn2018} the group \( \Pi(G) \leq \PGaL(X) \uparrow \sym(m) \) can be embedded into a wreath
product \( (G_X/\F^\times) \uparrow \pi(G) \), where \( G_X/\F^\times \) is the 1-component of \( \Pi(G) \).
In particular, \( G \) lies in the class \( \X \) if and only if \( G_X \) and \( \pi(G) \) lie in the class \( \X \).

\begin{proposition}\label{manytens}
	Let \( X \) be a vector space over a finite field \( \F \) with \( \dim X \geq 2 \). Let \( G \) and \( H \)
	be subgroups of \( (\GL(X) \otimes \dots \otimes \GL(X)) \rtimes (\aut(\F) \times \sym(m)) \), \( m \geq 2 \),
	preserving the \( m \)-fold tensor power \( X \otimes \dots \otimes X \), and assume
	that \( G \), \( H \) contain \( \F^\times \) and \( \pi(G) \), \( \pi(H) \) are transitive
	subgroups of \( \sym(m) \). Suppose that \( G \) and \( H \) are \( k \)-equivalent for some \( k \geq 3 \).
	Then \( G_X \) and \( H_X \) are \( k \)-equivalent, and \( \pi(G)^{[r]} = \pi(H)^{[r]} \) for \( r = \min \{|\mathrm{Orb}_k(G_X /\F^\times)|,\, m \} \),
	where \( G_X/\F^\times \) acts on the set of lines of \( X \) and \( \mathrm{Orb}_k(G_X / \F^\times) \) is the set of orbits on \( k \)-tuples
	in this action.
\end{proposition}
\begin{proof}
	We will use the following notation throughout the proof. Given \( s \) simple tensors \( v_i = v_{i1} \otimes \dots \otimes v_{im} \),
	\( i = 1, \dots, s \), let \( T(v_1, \dots, v_s) \) denote the \( s \times m \) matrix:
	\[ T(v_1, \dots, v_s) =
	\begin{pmatrix}
		\F \cdot v_{11} & \dots & \F \cdot v_{1m}\\
		\vdots & & \vdots\\
		\F \cdot v_{s1} & \dots & \F \cdot v_{sm}\\
	\end{pmatrix}
	\]
	Since each entry of this matrix is a line in \( X \), the matrix \( T(v_1, \dots, v_s) \) is defined correctly
	and does not depend on the choice of \( v_{i1}, \dots, v_{im} \) in the decomposition of \( v_i \), \( i = 1, \dots, s \).
	\smallskip

	\textit{Claim: \( G_X \) is \( k \)-equivalent to \( H_X \).}
	We will prove that \( \krel(H_X) \subseteq \krel(G_X) \). 
	Let \( x_1, \dots, x_k \in X \) be arbitrary nonzero vectors, and let \( g_X \in G_X \) be an arbitrary element.
	Choose \( g \in G \) such that \( g = (g_1 \otimes \cdots \otimes g_m)\sigma\tau \), where \( \sigma \in \aut(\F) \),
	\( \tau \in \sym(m) \), \( \tau(1) = 1 \) and \( g_X = g_1\sigma \).
	Assume first that \( x_1, \dots, x_k \) lie on the same line, i.e.\ for some nonzero \( u \in X \)
	we have \( x_i \in \F^\times\cdot u \) for all \( i = 1, \dots, k \).

	Let \( \omega \) be a cyclic generator of \( \F^\times \). Since \( \dim X \geq 2 \) we can find
	a vector \( w \in X \) which is linearly independent from \( u \).
	Setting \( v = u^{g_X} \) we have \( (\omega \cdot u)^{g_X} = \omega^\sigma \cdot v \).
	Now consider three simple tensors: \( a = u \otimes w \otimes \dots \otimes w \),
	\( b = \omega \cdot u \otimes w \otimes \dots \otimes w \) and \( c = w \otimes w \otimes \dots \otimes w \),
	which differ only in the first component. By the definition of \( G_X \), the element \( g_X \) stabilizes
	the first component, therefore
	\begin{align*}
		a^g = v \otimes & w^{g_{\tau^{-1}(2)}\sigma} \otimes \dots \otimes w^{g_{\tau^{-1}(m)}\sigma},\\
		b^g = \omega^\sigma \cdot v \otimes & w^{g_{\tau^{-1}(2)}\sigma} \otimes \dots \otimes w^{g_{\tau^{-1}(m)}\sigma},\\
		c^g = w^{g_X} \otimes & w^{g_{\tau^{-1}(2)}\sigma} \otimes \dots \otimes w^{g_{\tau^{-1}(m)}\sigma}.
	\end{align*}
	It is readily seen that each column of the matrix \( T(a^g, b^g, c^g) \) has all coordinates equal to each other
	with the exception of the first column. The first column does not have all coordinates equal to each other,
	as vectors \( v \), \( \omega^\sigma \cdot v \) and \( w^{g_X} \) do not lie on the same line.

	Now, groups \( G \) and \( H \) are \( k \)-equivalent, in particular, they are \( 3 \)-equivalent,
	so there exists an element \( h \in H \) such that \( a^g = a^h \), \( b^g = b^h \) and \( c^g = c^h \).
	Write \( h = (h_1 \otimes \dots \otimes h_m)\phi\mu \), where \( \phi \in \aut(\F) \) and \( \mu \in \sym(m) \).
	It follows that
	\begin{align*}
		a^h &= w^{h_{\mu^{-1}(1)}\phi} \otimes \dots \otimes u^{h_1\phi} \otimes \dots \otimes w^{h_{\mu^{-1}(m)}\phi},\\
		b^h &= w^{h_{\mu^{-1}(1)}\phi} \otimes \dots \otimes \omega^\phi \cdot u^{h_1\phi} \otimes \dots \otimes w^{h_{\mu^{-1}(m)}\phi},\\
		c^h &= w^{h_{\mu^{-1}(1)}\phi} \otimes \dots \otimes w^{h_1\phi} \otimes \dots \otimes w^{h_{\mu^{-1}(m)}\phi},
	\end{align*}
	where the ``central'' component of each tensor is on position~\( \mu(1) \). Since \( u^{h_1\phi} \),
	\( \omega^\phi \cdot u^{h_1\phi} \) and \( w^{h_1\phi} \) do not lie on the same line, each column of the matrix
	\( T(a^h, b^h, c^h) \) has equal coordinates with the exception of column number~\( \mu(1) \).
	Since \( T(a^g, b^g, c^g) = T(a^h, b^h, c^h) \), that must be the first column and hence \( \mu(1) = 1 \).

	Equality \( a^g = a^h \) of simple tensors implies that there exist scalars \( \alpha_i \in \F^\times \),
	\( i = 1, \dots, m \), such that \( v = \alpha_1 \cdot u^{h_1\phi} \) and
	\( w^{g_{\tau^{-1}(i)}\sigma} = \alpha_i \cdot w^{h_{\mu^{-1}(i)}\phi} \) for \( i = 2, \dots, m \),
	and \( \alpha_1 \cdot \alpha_2 \cdot \dots \cdot \alpha_m = 1 \).
	Similarly, \( b^g = b^h \) implies that there exist scalars \( \beta_i \in \F^\times \),
	\( i = 1, \dots, m \), with \( \omega^\sigma \cdot v = \beta_1\omega^\phi \cdot u^{h_1\phi} \) and
	\( w^{g_{\tau^{-1}(i)}\sigma} = \beta_i \cdot w^{h_{\mu^{-1}(i)}\phi} \) for \( i = 2, \dots, m \),
	and \( \beta_1 \cdot \dots \beta_m = 1 \).
	Since the vector \( w \) is nonzero, equalities
	\begin{align*}
		w^{g_{\tau^{-1}(i)}\sigma} &= \alpha_i \cdot w^{h_{\mu^{-1}(i)}\phi},\\
		w^{g_{\tau^{-1}(i)}\sigma} &= \beta_i \cdot w^{h_{\mu^{-1}(i)}\phi}
	\end{align*}
	imply \( \alpha_i = \beta_i \) for \( i = 2, \dots, m \). It follows that \( \alpha_1 = \beta_1 = \alpha \)
	and hence \( v = \alpha \cdot u^{h_1\phi} \) and \( \omega^\sigma \cdot v = \alpha\omega^\phi \cdot u^{h_1\phi} \).
	We can choose an element \( z \in \F^\times \) such that \( v = u^{zh_1\phi} \) and
	\( \omega^\sigma \cdot v = \omega^\phi \cdot u^{zh_1\phi} \). We derive \( \omega^\sigma = \omega^\phi \),
	and since \( \omega \) is a cyclic generator of \( \F^\times \), we have \( \sigma = \phi \).
	Finally, as \( \F^\times \leq H_X \) and \( \mu(1) = 1 \), the element \( h'_X = zh_1\sigma \) lies in \( H_X \).
	To sum up, \( u^{g_X} = v = u^{h'_X} \) and \( (\alpha \cdot u)^{g_X} = \alpha^\sigma \cdot v = (\alpha \cdot u)^{h'_X} \)
	for any \( \alpha \in \F^\times \).

	Recall that \( x_i \in \F^\times \cdot u \), \( i = 1, \dots, k \). There exist scalars \( \alpha_i \in \F^\times \)
	such that \( x_i = \alpha_i \cdot u \), \( i = 1, \dots, k \). It follows that for any \( i \in \{ 1, \dots, m \} \)
	\[ x_i^{g_X} = (\alpha_i \cdot u)^{g_X} = (\alpha_i \cdot u)^{h'_X} = x_i^{h'_X}, \]
	and this finishes the proof of the case when all \( x_i \) lie on the same line.

	Now assume \( x_1, \dots, x_k \in X \) do not lie on the same line. Choose some nonzero vector \( w \in X \)
	and define simple tensors \( v_i = x_i \otimes w \otimes \dots \otimes w \), \( i = 1, \dots, k \),
	which differ only in the first component. Arguing as before, we see that each column of the matrix
	\( T(v_1^g, \dots, v_k^g) \) has equal coordinates, with the exception of the first column.
	The first column does not have all coordinates equal to each other as \( x_i^{g_X} \), \( i = 1, \dots, k \),
	do not lie on the same line.

	Since \( G \) and \( H \) are \( k \)-equivalent, there exists an element \( h \in H \) such that
	\( v_i^g = v_i^h \) for all \( i = 1, \dots, k \). Write \( h = (h_1 \otimes \dots \otimes h_m)\phi\mu \),
	where \( \phi \in \aut(\F) \), \( \mu \in \sym(m) \). All columns of the matrix \( T(v_1^h, \dots, v_k^h) \)
	have equal coordinates with the exception of column number \( \mu(1) \), which has coordinates
	\( x_1^{h_1\phi}, \dots, x_k^{h_1\phi} \). Since \( T(v_1^g, \dots, v_k^g) = T(v_1^h, \dots, v_k^h) \),
	we derive \( \mu(1) = 1 \).

	Equality \( v_i^g = v_i^h \), \( i = 1, \dots, k \), of simple tensors implies that there exist scalars
	\( \alpha_{ij} \in \F^\times \) such that for every \( i \in \{ 1, \dots, k \} \) we have
	\( x_i^{g_X} = \alpha_{i1} \cdot x_i^{h_1\phi} \) and \( w^{g_j\sigma} = \alpha_{ij} \cdot w^{h_j\phi} \)
	for \( j = 2, \dots, m \), and \( \prod_{j = 1}^m \alpha_{ij} = 1 \). Since the vector \( w \) is nonzero,
	scalars \( \alpha_{ij} \), \( j = 2, \dots, m \), do not depend on \( i \). It follows that \( \alpha_{i1} \)
	does not depend on \( i \) and we can set \( \alpha = \alpha_{i1} \) for all \( i = 1, \dots, k \).

	We have \( x_i^{g_X} = \alpha \cdot x_i^{h_1\phi} \) for \( i = 1, \dots, k \). Choose an element \( z \in \F^\times \)
	such that \( x_i^{g_X} = x_i^{zh_1\phi} \) for \( i = 1, \dots, k \). Since \( \F^\times \leq H_X \) and \( \mu(1) = 1 \), the element
	\( h'_X = zh_1\phi \) lies in \( H_X \) and we have \( x_i^{g_X} = x_i^{h'_X} \), \( i = 1, \dots, k \).
	This finishes the proof of the case when not all \( x_i \), \( i = 1, \dots, k \), lie on the same line.

	It follows that \( \krel(H_X) \subseteq \krel(G_X) \), and by repeating the same argument for \( G \)
	and \( H \) with interchanged roles, we derive \( \krel(H_X) = \krel(G_X) \), i.e.\ \( G_X \) and \( H_X \)
	are \( k \)-equivalent. The first claim is proved.
	\smallskip

	\textit{Claim: \( \pi(G)^{[r]} = \pi(H)^{[r]} \).}
	By Proposition~\ref{basicEmb}~(iii) it is sufficient to prove that \( \Pi(G) \) is \( k \)-equivalent to \( \Pi(H) \).
	Let \( (\F v_1, \dots, \F v_k) \) be a \( k \)-tuple of lines for some \( v_1, \dots, v_k \in X \otimes \dots \otimes X \).
	Take some \( g \in \Pi(G) \) and let \( g' \in G \) be such that \( g = \Pi(g') \).
	Define \( u_i = v_i^{g'} \) for \( i = 1, \dots, k \). It follows that
	\[ (\F v_1, \dots, \F v_k)^g = (\F v_1, \dots, \F v_k)^{g'} = (\F u_1, \dots, \F u_k). \]
	Since \( G \) and \( H \) are \( k \)-equivalent, there exists \( h' \in H \) such that \( u_i = v_i^{h'} \), \( i = 1, \dots, k \).
	Setting \( h = \Pi(h) \in \Pi(H) \) we have
	\[ (\F v_1, \dots, \F v_k)^h = (\F v_1, \dots, \F v_k)^{h'} = (\F u_1, \dots, \F u_k) = (\F v_1, \dots, \F v_k)^g. \]
	Hence \( \Pi(G) \) and \( \Pi(H) \) are \( k \)-equivalent by Lemma~\ref{krellem} and the claim is proved.
\end{proof}

\begin{corollary}\label{clPower}
	Let \( V \) be a vector space over a finite field \( \F \), and let \( G \) be a primitive affine permutation
	group with socle \( V \). Set \( G = V \rtimes G_0 \), where \( G_0 \leq \GaL(V) \), and assume that \( \F^\times \leq G_0 \).
	Suppose that \( G_0 \) preserves a tensor decomposition \( V = X \otimes \dots \otimes X \), where \( X \) is a vector
	space over \( \F \) and there are \( m \geq 2 \) tensor factors in the decomposition of \( V \).
	Assume that \( k \geq 4 \) and the \( k \)-closures of \( X \rtimes (G_0)_X \) and \( \pi(G_0) \) lie in the class \( \X \).
	Then \( G^{(k)} \) lies in the class \( \X \).
\end{corollary}
\begin{proof}
	By Lemma~\ref{4claff}, the \( k \)-closures of \( G \) and \( X \rtimes (G_0)_X \) are affine groups with socles \( V \)
	and \( X \), respectively. Set \( H = G^{(k)} \) and \( H = V \rtimes H_0 \), so \( G_0 \) is \( (k-1) \)-equivalent to \( H_0 \).
	By Proposition~\ref{xukcl}~(i) and~\cite[Lemma~4.4.5~(2)]{XuGLP2011}, \( H_0 \) preserves the tensor decomposition \( V = X \otimes \dots \otimes X \),
	and by Proposition~\ref{manytens}, the group \( (G_0)_X \) is \( (k-1) \)-equivalent to \( (H_0)_X \),
	and \( \pi(G_0)^{[r]} = \pi(H_0)^{[r]} \), where \( r = \min \{ |\mathrm{Orb}_{k-1}((G_0)_X / \F^\times)|,\, m\} \).
	By Lemma~\ref{affcl} and our assumptions, \( (H_0)_X \) lies in the class \( \X \).

	Now, as \( k-1 \geq 3 \), we have \( |\mathrm{Orb}_{k-1}((G_0)_X / \F^\times)| \geq k+1 \) since \( X \) contains at least 3 lines.
	If \( m < k+1 \), then \( r = m \) and \( \pi(G_0)^{[r]} = \pi(G_0)^{[m]} = \pi(G_0) \), where the last equality follows from the definition of the closure with
	respect to partitions. If \( m \geq k+1 \), then \( r \geq k+1 \) and \( \pi(G_0)^{[r]} \leq \pi(G_0)^{[k+1]} \leq \pi(G_0)^{(k)} \),
	where the last inclusion follows from Proposition~\ref{partClkCl}. In either case \( \pi(H_0) \leq \pi(G_0)^{[r]} \) lies in the class \( \X \), so the claim is proved.
\end{proof}

\section{Proof of Theorem~\ref{t:main}}\label{secProof}

Set \( H = G^{(4)} \).
By Proposition~\ref{basicProps}~(i) it suffices to prove that \( H \in \X \).
Suppose that \( G \) is a counterexample to Theorem~\ref{t:main} of minimal degree~\( n \).
Note that \( n > 24 \), because otherwise \( H \) is \( \alt(25) \)-free and hence lies in \( \X \).
By~\cite[Theorem~3.1~(i)]{PV2024} (see also Proposition~\ref{basicEmb}), we may assume that \( G \) is a primitive basic permutation group.

Since $G$ is primitive, \cite[Theorem~2]{PS1992} implies that
\begin{equation}\label{socles}
	\soc(G)=\soc(H)
\end{equation}
unless $G$ is $4$-transitive. By \cite[Tables~7.3 and~7.4]{CB}), if $G$ is $4$-transitive and $n>24$, then $G\geq\alt(n)$, so \( \alt(n) \in \X \)
and hence \( H = \sym(n) \in \X \), a contradiction.
The case \( n \leq 24 \) is not possible, so we may assume that equality~\eqref{socles} holds.
\medskip

Since $G$ is basic, applying the O'Nan-Scott theorem~\cite{LPS1988b}, we arrive to the following three cases:
\begin{enumerate}[(i)]
	\item $G$ is almost simple,
	\item $G$ is in a {\em diagonal action},
	\item $G$ is an affine group.
\end{enumerate}

In the case~(i), equality \eqref{socles} implies that the group $H/\soc(G)$ is solvable and hence $H \in \X$, a contradiction.\medskip

Suppose that $G$ is in a diagonal action as in the case~(ii).
Then $S=\soc(G)=T^m$, where $T$ is a nonabelian simple group, and $G/S$ is a subgroup of $\out(T)\times L$,
where $L$ is the symmetric group of degree $m$ acting faithfully on the set simple factors of the socle by conjugation.
By~\eqref{socles}, $\soc(H)=S$, so $H/S$ is also a subgroup of $\out(T)\times L$. If $m\leq 4$ or $G/S$ includes the alternating subgroup of $L$,
then the nonabelian composition factors of $G$ and $H$ are the same, and we are done.
Otherwise, \( G \) has a base of size~2 in view of \cite[Theorem~1.1]{Faw}, and $G^{(3)} = G$ by Proposition~\ref{basicProps}~(vi). Hence \( H=G\in\X\), a contradiction.

Thus, we may assume that $G$ is an affine group.
Observe that in this case \( H \) is also an affine group by~\eqref{socles}, and a primitive basic group by Proposition~\ref{basicProps}~(ii)
and Proposition~\ref{basicEmb}~(iii).
\medskip

By~\cite{BPVV2021}, we may assume that \( G \) is not solvable.
Let \( V \) be the socle of \( G \) and let \( G_0 \leq \GL(V) \) be the zero stabilizer, so \( G = V \rtimes G_0 \).
We may write \( H = V \rtimes H_0 \) for the zero stabilizer \( H_0 \leq \GL(V) \). Moreover, \( H_0 = G_0^{(3)} \cap \GL(V) \) by Lemma~\ref{affcl}.

We wish to apply Aschbacher's theorem~\cite{A1984} to \( G_0 \) and~\( H_0 \).
By Proposition~\ref{xukcl}, \( G_0 \) and \( H_0 \) lie in the same Aschbacher classes.
Set \( |V| = p^s \), where \( p \) is a prime. Choose the minimal \( a \geq 1 \) dividing \( s \) such that \( G_0 \leq \GaL_a(q) \) where \( q = p^{s/a} \).
Note that \( a \geq 2 \) since \( G \) is not solvable.
By Proposition~\ref{basicProps}~(iv), we may also assume that \( \F_q^\times \leq G_0 \), since \( \F_q^\times \cdot G_0 \) still lies in the class~\( \X \).
Note that the corresponding affine group \( V \rtimes (\F_q^\times \cdot G_0) \) is a permutation group of degree \( n = |V| \),
so it is still a counterexample of minimal degree.

Since \( G_0 \) preserves the structure of the extension field \( \F_q \) of \( \F_p \) (class \( \C_3 \) for Aschbacher's theorem in \( \GL(V) \)),
the group \( H_0 \) also preserves this field and hence \( H_0 \leq \GaL_a(q) \).
Now, we use Aschbacher's classification inside \( \GaL_a(q) \). The group \( G_0 \) is irreducible (since \( G \) is primitive) and primitive as a linear group (since \( G \) is basic),
hence the same is true for \( H_0 \). As \( a \) is minimal, \( G_0 \) does not preserve any extension field of \( \F_q \), hence \( G_0 \) and thus \( H_0 \)
do not lie in class \( \C_3 \). If \( G_0 \) preserves a tensor product decomposition (class \( \C_4 \)), a subfield (class \( \C_5 \)),
or a tensor power (class \( \C_7 \)), then from the minimality of the counterexample \( G \) and Corollary~\ref{clTens}, Proposition~\ref{clSubf} or Corollary~\ref{clPower}
respectively, it follows that \( H \in \X \), a contradiction.

If \( \mathrm{SL}_a(q) \leq G_0 \), then \( \mathrm{SL}_a(q) \leq H_0 \leq \GaL_a(q) \). Hence \( G_0 \) and \( H_0 \) have the same nonabelian composition factors,
and \( H \in \X \), a contradiction. Therefore, Aschbacher's theorem in \( \GaL_a(q) \) applies, and we deduce that either \( G_0 \) and \( H_0 \)
normalize a subgroup of symplectic type (class \( \C_6 \)), or they stabilize a symplectic, unitary or orthogonal form (class \( \C_8 \)),
or they lie in class \( \C_9 \). If \( G_0 \) stabilizes a form, then by Proposition~\ref{xukcl}~(i), \( H_0 \) preserves the same form and hence we can
apply Aschbacher's theorem inside a symplectic, unitary or orthogonal group. Note that if a group preserving some form also stabilizes a tensor decomposition
or a subfield, then we have classified that group as belonging to \( \C_4 \), \( \C_7 \) or \( \C_5 \) on the previous step.
Thus, either \( G_0 \) and \( H_0 \) lie in \( \C_6 \) or \( \C_9 \),
or \( G_0 \) contains \( \mathrm{Sp}_a(q) \), \( \mathrm{SU}_a(q^{1/2}) \) or \( \Omega_a^{\pm}(q) \), depending on the type of the form \( G_0 \) stabilizes.
In the latter case \( H_0 \) stabilizes the same form and contains the same classical group as \( G_0 \). Hence \( H_0 \) and \( G_0 \) have the same nonabelian composition
factors, which implies \( H_0 \in \X \), a contradiction.

Therefore, to finish the proof we are left to consider the cases when \( G_0 \) and \( H_0 \) lie in class \( \C_6 \) or \( \C_9 \).

\subsection{\texorpdfstring{The class \( \C_6 \).}{The class C6.}}\label{ss:c6}

The following lemma is probably well-known.

\begin{lemma}\label{regorb}
	Let \( R \) be a nilpotent group of class at most two acting faithfully and irreducibly on a finite vector space~\( V \).
	If \( |R| \leq |V|^{1/2} \), then \( R \) has a regular orbit on~\( V \).
\end{lemma}
\begin{proof}
	Let \( x \in R \) be a nontrivial element. If \( x \) is an element of the center \( Z(R) \), then \( C_V(x) \) is an \( R \)-invariant subspace of \( V \),
	and since \( R \) is faithful and irreducible, \( C_V(x) = 0 \). If \( x \not\in Z(R) \), then there exists some \( g \in R \) such that \( [x, g] = x^{-1} x^g \) is nontrivial.
	It follows that \( C_V([x, g]) \supseteq C_V(x^{-1}) \cap C_V(x^g) \), and since \( [x, g] \in [R, R] \leq Z(R) \), we have \( C_V(x^{-1}) \cap C_V(x^g) = 0 \).
	Observe that \( |C_V(x^{-1})| = |C_V(x)| = |C_V(x^g)| \), hence \( |C_V(x)| \leq |V|^{1/2} \).

	Now, suppose that \( R \) does not have a regular orbit on \( V \). Then
	\[ |V| \leq \sum_{x \in R \setminus 1} |C_V(x)| < |R| \cdot |V|^{1/2}, \]
	hence \( |V|^{1/2} < |R| \), a contradiction.
\end{proof}

If \( G_0 \) lies in class \( \C_6 \), then it contains a normal \( r \)-subgroup \( R \)
for some prime \( r \neq p \) with the following properties (see~\cite[Section~4.6]{KL} for details).
We have \( |R/Z(R)| = r^{2m} \), \( m \geq 1 \), and if \( r \) is odd, then \( |Z(R)| = r \) and \( R \) is an extraspecial group, while if \( r = 2 \), then
\( |Z(R)| \leq 4 \) and \( R \) is either an extraspecial group or a central product of a cyclic group of order~4 and an extraspecial group.
The dimension of \( V \) over \( \F_q \) is \( a = r^m \), and \( R \) acts absolutely irreducibly on~\( V \).
Moreover, \( G_0 \) is contained in a subgroup \( N \leq \GaL_a(q) \),
such that \( R \) is normal in \( N \) and \( N/R \) is isomorphic to \( \mathrm{Sp}_{2m}(r) \) or to \( O_{2m}^\pm(2) \) if \( r = 2 \).
By Proposition~\ref{xukcl}~(i), the group \( H_0 \) also lies in the class \( \C_6 \). Note that the order of the subgroup of symplectic type \( R \)
can be determined from \( a \), so it is the same for both \( G_0 \) and \( H_0 \).

\begin{proposition}\label{c6}
	If \( m < 12 \), then \( G \) and \( H \) are \( \alt(25) \)-free.
	If \( m \geq 12 \), then \( G \) has a base of size at most \( 3 \), in particular, \( G = H \).
\end{proposition}
\begin{proof}
	Recall that \( G_0/R \) embeds into \( \GL_{2m}(r) \). If \( \alt(25) \) is a section of \( G \), then it is a section of \( \GL_{2m}(r) \) as well.
	By Proposition~\ref{altGL}, we have \( 2m \geq 23 \). This implies that if \( m < 12 \) then \( G \) (and similarly \( H \))
	is \( \alt(25) \)-free. So we may assume \( m \geq 12 \).

	Suppose that \( |R| > |V|^{1/2} \).
	Recall that \( |V| = q^{r^m} \), \( r \leq q-1 \) and \( |R| \leq r^{2m+2} \), therefore
	\[ q^{r^m/2} = |V|^{1/2} < |R| \leq r^{2m+2} \leq q^{2m+2}. \]
	This is impossible for \( m \geq 12 \), so \( |R| \leq |V|^{1/2} \) and \( R \) has a regular orbit on~\( V \) by Lemma~\ref{regorb}.

	Let \( v \in V \) be such that the stabilizer \( R_v \) is trivial and set \( J = (G_0)_v \). We will now prove that \( J \) has a regular orbit on \( V \),
	our argument following the argument in \cite[Lemma~3.6]{LS} (see also \cite[Proposition~5.6]{HLM}).

	Since \( R \cap J = 1 \), the group \( J \) is isomorphic to a subgroup of $\mathrm{Sp}_{2m}(r)$ or \( O^\pm_{2m}(2) \).
	Suppose that \( J \) has no regular orbit and therefore
	\begin{equation}\label{eqvup}
		|V| \leq \sum_{h \in J \setminus 1} |C_V(h)|.
	\end{equation}
	For \( h \in J \setminus 1 \), \cite[Theorem~4.1]{GS} implies that there are \( 2m+2 \) conjugates of \( h \) in \( N \) generating~\( N \).
	Hence
	\[ \dim C_V(h) \leq \left( 1 - \frac{1}{2m+2} \right) \dim V \]
	and \( |C_V(h)| \leq |V|^{1-1/(2m+2)} \). The obtained inequality and~\eqref{eqvup} imply
	\[ |V| \leq |J| \cdot |V|^{1-1/(2m+2)}. \]
	We derive that
	\begin{equation}\label{eqemb}
		q^{\frac{r^m}{2m+2}} = |V|^{1/(2m+2)} \leq |J| \leq |N/R| \leq |\GL_{2m}(r)| < r^{4m^2}.
	\end{equation}
	Since \( r \leq q-1 \), we obtain \( r^m < 4m^2(2m+2) \). The only integer solutions for that inequality for \( m \geq 12 \) and \( r \geq 2 \)
	are \( r = 2 \) and \( m = 12, 13, 14 \).

	Note that \( q \geq 3 \). Inequalities~\eqref{eqemb} imply
	\[ q^{\frac{2^m}{2m+2}} \leq |N/R| \leq \max \{ |\mathrm{Sp}_{2m}(2)|,\, |O^\pm_{2m}(2)| \}. \]
	A direct computation shows that this is possible only for \( q = 3 \) and \( m = 12 \).

	We are left with the case \( r = 2 \), \( q = 3 \), \( m = 12 \). By~\cite[Table~4.6.B]{KL},
	we have \( N/R \simeq O^\varepsilon_{24}(2) \), where $\varepsilon\in\{+,-\}$. Recall that \( J \) is isomorphic to a subgroup of \( O^\varepsilon_{24}(2) \),
	and consider an arbitrary element \( h \in J \setminus 1 \).
	By~\cite[Theorem~4.4]{GS}, there are \( m+3 = 15 \) conjugates of \( h \) which generate \( O^\varepsilon_{24}(2) \) unless \( h \) is a transvection,
	in which case there are \( 2m = 24 \) conjugates.
	There are at most \( 2^{24} \) transvections in \( O^\varepsilon_{24}(2) \), so at most \( 2^{24} \) elements \( h \in J \setminus 1 \)
	corresponding to a transvection; denote this set of elements by \( T \).

	Now, if \( h \in T \), there are \( 25 \) conjugates of \( h \) generating \( N \), hence
	\[ \dim C_V(h) \leq \left( 1 - \frac{1}{25} \right) \dim V \]
	and \( |C_V(h)| \leq |V|^{1-1/25} \).
	If \( h \not\in T \), there are \( 16 \) conjugates of \( h \) generating \( N \), so
	\[ \dim C_V(h) \leq \left( 1 - \frac{1}{16} \right) \dim V \]
	and \( |C_V(h)| \leq |V|^{1-1/16} \). We apply our estimates to inequality~\eqref{eqvup} and obtain
	\[ |V| \leq |T| \cdot |V|^{1-1/25} + (|J| - |T|) \cdot |V|^{1-1/16}. \]
	Recall that \( |T| \leq 2^{24} \), \( |J| \leq |O^-_{24}(2)| \) and \( |V| = 3^{2^{12}} \). Thus
	\[ 1 \leq 2^{24} \cdot 3^{-\frac{2^{12}}{25}} + |O^-_{24}(2)| \cdot 3^{-\frac{2^{12}}{16}}, \]
	which is a contradiction. Hence for any \( m \geq 12 \) the group \( J \) has a regular orbit
	and so \( G \) has a base of size at most~3. Now \( G = H \) by Proposition~\ref{basicProps}~(vi).
\end{proof}

The proposition above contradicts \( G \) being a counterexample, so we can move to the next class.

\subsection{\texorpdfstring{The class $\C_9$.}{The class C9.}}\label{ss:c9}

We need the following general statement about permutation groups.

\begin{proposition}\label{c9}
	Let \( P \) and \( Q \) be permutation groups on the same set \( V \),
	and suppose that they are \( 3 \)-equivalent. If \( P/Z(P) \) and \( Q/Z(Q) \) are almost simple groups, then
	either \( \soc(P/Z(P)) = \soc(Q/Z(Q)) \) or \( P \) and \( Q \) are \( \alt(25) \)-free.
\end{proposition}
\begin{proof}
	Recall that since \( P \) and \( Q \) are 3-equivalent, they have the same orbits and systems of imprimitivity by Proposition~\ref{basicProps}~(ii).
	There exists an orbit \( \Delta \subset V \) such that \( P^\Delta \) is nonsolvable.
    Groups \( P^\Delta \) and \( Q^\Delta \) are 3-equivalent by Proposition~\ref{basicEmb}~(i),
    hence \( Q^\Delta \) is also nonsolvable (otherwise \( P^\Delta \leq (Q^\Delta)^{(3)} \) would be solvable by~\cite{BPVV2021}).
	The group \( P^\Delta/Z(P^\Delta) \) is almost simple with the same socle as \( P \), and similar properties hold for \( Q^\Delta \).
	We may thus replace \( P \) and \( Q \) by \( P^\Delta \) and \( Q^\Delta \), and assume that \( P \) is transitive on~\( \Delta \).

	If \( P \) is imprimitive on \( \Delta \), then let \( \Pi \subset \Delta \) be the minimal nontrivial block of imprimitivity.
	If \( P^\Pi \) is nonsolvable, then \( P^\Pi/Z(P^\Pi) \) is almost simple with the same socle as \( P \), and hence again, we can
	replace \( P \) and \( Q \) by \( P^\Pi \) and \( Q^\Pi \) by Proposition~\ref{basicEmb}~(ii). If \( P^\Pi \) is solvable, then the action of \( P \) on the
	imprimitivity system \( \Delta/\Pi \) defined by \( \Pi \) is nonsolvable and almost simple, therefore we can replace our groups
	by \( P^{\Delta/\Pi} \) and \( Q^{\Delta/\Pi} \) by Proposition~\ref{basicEmb}~(ii). We can thus assume that \( P \) acts primitively on~\( \Delta \).

	Since \( P \) is a nonsolvable primitive group, \( Z(P) = 1 \). Hence \( P \) is almost simple and by~\cite[Theorem~2]{LPS88}, either
	\( \soc(P) = \soc(P^{(3)}) \) or \( P \) is 3-transitive. The same is true for \( Q \), and as \( P \) and \( Q \) are 3-equivalent, \( P \)
	is 3-transitive if and only if \( Q \) is 3-transitive. Since \( P^{(3)} = Q^{(3)} \), we either have \( \soc(P) = \soc(P^{(3)}) = \soc(Q^{(3)}) = \soc(Q) \),
	or \( P \) and \( Q \) are 3-transitive. In the latter case, \cite[Table~7.4]{CB} implies that \( |\Delta| \leq 24 \) or
	\( \soc(Q) = \mathrm{PSL}_2(u) \) for some prime power \( u \), which means that \( P \) and \( Q \) are \( \alt(25) \)-free.
\end{proof}

Let \( G_0 \) lie in \( \C_9 \). Since \( n > 16 \), the group \( H_0 \) also lies in \( \C_9 \) by Proposition~\ref{xukcl}~(ii).
Both \( G_0/Z(G_0) \) and \( H_0/Z(H_0) \) are almost simple groups.
By applying Proposition~\ref{c9} to \( P = G_0 \) and \( Q = H_0 \) we derive that \( H_0 \in \X \), which is the final contradiction.

\section{Concluding remarks}\label{s:end}

\subsection{}
As mentioned in the introduction, it is important from the computational point of view to have an upper bound on
the size of primitive sections of a permutation group from the class under consideration. Combining the main result of~\cite{BCP1982} with
our Corollary~\ref{c:main}, we obtain the following.

\begin{corollary}\label{c:orders}
	Suppose that \( G \) is an $\alt(d)$-free primitive permutation group with \( d \geq 25 \).
	Then the order of \( G^{(k)} \) is bounded from above by $n^c$ for every \( k \geq 4 \), where $c$ is a constant depending only on~$d$.
\end{corollary}

\subsection{} Example (ii) from the introduction shows that for $k=4$ (and even $k=5$), the best possible value of $d$ providing that the $k$-closure of
every $\alt(d)$-free group is again $\alt(d)$-free is $d=25$. It is clear that for greater value of $k$, the value of $d$ can be reduced. It
would be interesting to find the minimal possible $d$ as a function on $k$, or at least to find the minimal possible $k$ such that the
$k$-closure of $\alt(d)$-free group is $\alt(d)$-free for all $d\geq5$.

\subsection{}
The concept of the $\alt(d)$-free groups provides a natural way to restrict the composition factors of a group.
Namely, only nonabelian composition factors that do not have linear representations of bounded dimension must be excluded. For example, if
$\X$ is the class of $\alt(25)$-free groups, then the list of simple groups of $\X$ includes the groups of order $p$ for all primes $p$, all
the sporadic groups and exceptional groups of Lie type, as well as the classical groups of bounded degree (cf.\ Proposition~\ref{altGL}), and
the alternating groups of degree less than~$25$.

It is worth noting that historically, there were other ways to restrict composition factors of a group, see the classical papers by Luks
\cite{L1982} and Babai--Luks \cite{BL1983}.  Namely, let $\Gamma_d$ with $d\geq4$ denote the class of finite groups every nonabelian
composition factor of which can be embedded in $\alt(d)$. In fact, the restriction in the Luks paper was even stronger: it requires that
every (not necessarily nonabelian) composition factor can be embedded in $\alt(d)$. In our definition of $\Gamma_d$ we follow the Babai--Luks
paper, since it provides a wider class of groups.

Obviously, every group from $\Gamma_d$ is $\alt(d+1)$-free. Therefore, if $k\geq4$ and $d\geq24$, then $G^{(k)}$ is an $\alt(d+1)$-free
group for every $G\in\Gamma_d$ by virtue of  Corollary~\ref{c:main}. The question is whether $G^{(k)}\in\Gamma_d$?

Inspecting the proof of Theorem~\ref{t:main}, one can see that applying our arguments to a group $G\in\Gamma_d$ for $d\ge24$, it follows
that for $k\geq4$, either $G^{(k)}$ belongs to $\Gamma_d$, or $G$ is a basic affine group such that the zero stabilizer $G_0$ lies in one of
the Aschbacher classes $\C_6$ and~$\C_9$ (as in Subsections~\ref{ss:c6} and~\ref{ss:c9}).

For $G_0\in\C_9$, it is not hard to show that $G^{(4)}\in\Gamma_d$. Here the only problem is (cf.\ the proof of Proposition~\ref{c9}) the
case when $G_0$ is $3$-transitive. Since $d\ge24$, we may assume that $\soc(G_0)=PSL_2(u)$. Applying the results from \cite{Lee1,Lee2} that
describe the irreducible quasisimple linear groups acting without regular orbits, we conclude that for a finite number of exceptions
$b(G_0)\leq2$, and we are done in view of Proposition~\ref{basicProps}(vi).

For $G_0\in\C_6$, the situation is more complicated. Let us keep notation of Proposition~\ref{c6}. If $m\geq12$, then $G^{(4)}=H^{(4)}$ and
everything is fine. However, the inequality $m<12$ on dimension of $G_0$ can only guarantee that $m<12$ for $H_0$, that is the dimension
(but not the size) of $H_0$ is bounded. We do not know how to overcome this problem for $k=4$. Nevertheless, it can be done for $k\geq15$.
Indeed, according to \cite[Proposition~5.6]{HLM}, in our situation, for each $m$ (even $m<12$), the base number of $G_0$ is at most~$13$, so
the base number of $G$ is at most $14$. Now Proposition~\ref{basicProps}(vi) implies that $G^{(15)}=H^{(15)}$, as required. Thus, the
following is true.

\begin{corollary}\label{c:gammad}
	If $G\in\Gamma_d$ with $d\geq24$, then $G^{(k)}\in\Gamma_d$ for $k\geq15$.
\end{corollary}

It would certainly be interesting to prove the similar statement for $k<15$ if it is possible.

\subsection{}
Let us finish by returning to the computational problem from which we started the introduction. In the context of
the present paper, it can be formulated as follows: find the $k$-closure of a given $\alt(d)$-free permutation group in polynomial time in
its degree for fixed $k$ and~$d$. Using~\cite[Theorem 3.1]{PV2024},  this problem can be reduced in polynomial time to finding (a) the
$k$-closure of every primitive basic $\alt(d)$-free group,  and (b) the intersection of the $k$-closure of an $\alt(d)$-free group with any
other $\alt(d)$-free group. Problem~(b) can be solved by the Babai--Luks algorithm for all $k$ and $d$. Problem (a) is open for now. We believe that the results of this paper could be used to solve problem~(a).

\section{Acknowledgements}

The authors thank A.~Mar\'oti for helpful suggestions, improving the treatment of class~\( \C_6 \) in Subsection~\ref{ss:c6}.

The work is supported by the Russian Science Foundation, project no.~24-11-00127, \url{https://rscf.ru/en/project/24-11-00127/}.

\end{document}